\documentstyle[amssymb,amsfonts]{amsart}

\def\mod{{\hbox{\rm mod}}}
\def\R{{\hbox{\bf R}}}

 at 10 true pt

\def\be#1{ \begin{equation}\label{#1} }

\def\bas{\begin{align*}}
\def\eas{\end{align*}}
\def\bi{\begin{itemize}}
\def\ei{\end{itemize}}

\def\Z{{\hbox{\bf Z}}}

\newenvironment{proof}{\noindent {\bf Proof} }{\endprf\par}
\def \endprf{\hfill  {\vrule height6pt width6pt depth0pt}\medskip}
\def\emph#1{{\it #1}}
\def\textbf#1{{\bf #1}}

\def\ep{{\epsilon}}

\def\mod\quad {\hbox {\rm mod}}

\def \ep{\varepsilon}

\theoremstyle{plain}
  \newtheorem{theorem}[subsection]{Theorem}
  \newtheorem{question}[subsection]{Question}

  \newtheorem{proposition}[subsection]{Proposition}
  \newtheorem{fact}[subsection]{Fact}
  \newtheorem{lemma}[subsection]{Lemma}
  
  \newtheorem{example}[subsection]{Example}

\theoremstyle{remark}
  \newtheorem{remark}[subsection]{Remark}

\theoremstyle{definition}
  \newtheorem{definition}[subsection]{Definition}

\include{psfig}

\begin{document}

\title[Squares in sumsets]
{Squares in sumsets}

\author{}
\author{Hoi H. Nguyen}
\address{Department of Mathematics, Rutgers University, Piscataway, NJ 08854, USA}
\email{hoi@@math.rutgers.edu}
\thanks{}

\author{Van H. Vu}
\address{Department of Mathematics, Rutgers University, Piscataway, NJ 08854, USA}
\email{vanvu@@math.rutgers.edu}

\thanks{The authors are partially  supported by an NSF Career Grant.}

\begin{abstract} A finite set $A$ of integers  is square-sum-free if there is no subset of $A$ sums up to a square. In 1986, Erd\H os posed the problem of determining the largest cardinality of a square-sum-free subset of $\{1, \dots, n \}$. Answering this question, we show that this maximum cardinality is of order  $n^{1/3+o(1)}$.
\end{abstract}

\maketitle

\section{Introduction}

\noindent Let $A$ be a set of numbers. We denote by $S_A$ the collection of
finite partial sums of $A $,

$$
S_A := \left \{ \sum_{x\in B}x ; B\subset A, 0 < |B| < \infty \right \}.
$$

\noindent For a positive integer $l\le |A|$ we denote by $l^\ast A$ the
collection of partial sums of $l$ elements of $A$,

$$
l^\ast A := \left \{ \sum_{x\in B}x ; B\subset A, |B|=l \right \}.
$$

\noindent  Let $[x]$ denote the set of positive integers at most $x$. In 1986, Erd\H{o}s \cite{Erdos} raised the following question:

\begin{question} \label{question:1}
What is the maximal cardinality of a subset $A$ of $[n]$
such that $S_A$ contains no square?
\end{question}

\noindent  We denote by $SF(n)$ the maximal cardinality in
question. Erd\H{o}s observed that

\begin{equation} \label{eqn:lowerbound} SF(n)=\Omega(n^{1/3}). \end{equation}

\noindent To see this, consider the following example

\begin{example}\label{example:1} Let $p$ be a prime and $k$ be the largest integer such that $kp \le n$. We choose $p$ of order $n^{2/3}$ such that
$k=\Omega (n^{1/3})$ and $1+\dots +k <p$. Then the set
$A:=\{p,2p, \dots,  k p \}$ is square-sum-free.
\end{example}

\begin{remark}  \label{remark:example1} The fact that $p$ is a prime is not essential. The construction still works if we choose $p$ to be a square-free number, namely, a number of the form $p=p_1\dots p_l$ where $p_i$ are different primes. \end{remark}

\noindent Erd\H os  \cite{Erdos} conjectured that $SF(n)$ is close to the lower bound in \eqref{eqn:lowerbound}. Shortly after Erd\H os' paper,   Alon \cite{Alon} proved the first non-trivial upper bound

\begin{equation} \label{eqn:upper1} SF(n)=O(\frac{n}{\log n}). \end{equation}

\noindent Next,  Lipkin \cite{Lipkin} improved to

\begin{equation} \label{eqn:upper2} SF(n)=O(n^{3/4+o(1)}). \end{equation}

\noindent In
\cite{AlonFreiman}, Alon and Freiman improved the bound further to

\begin{equation} \label{eqn:upper3} SF(n)=O(n^{2/3+o(1)}). \end{equation}

\noindent The latest development was due to
S\'ark\"ozy \cite{Sarkozy}, who showed

\begin{equation} \label{eqn:upper4} SF(n) = O(\sqrt {n\log n} ). \end{equation}

\noindent In this paper, we obtain the asymptotically tight bound

 \begin{equation} \label{eqn:upper5}  SF(n)= O(n^{1/3+o(1)}). \end{equation}

\begin{theorem}\label{theorem:square:1}
There is a constant $C$ such that for all $n \ge 2$

\begin{equation} \label{eqn:upper6} SF(n) \le n^{1/3} (\log n)^C \end{equation}

\end{theorem}

In fact, we are going to prove the following (seemingly) more general theorem

\begin{theorem}\label{theorem:square:p}
There is a constant $C$ such that the following holds for all sufficiently large $n$.  Let  $p$ be positive
integer less than $n^{2/3}(\log n)^{-C}$ and  $A$ be a subset of
cardinality $n^{1/3}(\log n)^C$ of $[n/p]$. Then there exists an integer
$z$ such that $pz^2 \in S_A.$
\end{theorem}

\noindent Theorem \ref{theorem:square:1} is the special case when $p=1$. Furthermore, Theorem \ref{theorem:square:1} implies many special cases of  Theorem \ref{theorem:square:p}. To see this, choose $A$ to have the form $A := \{p b \,\,| b \in B \}$ where $B$ is a subset of $[n/p]$ and $p$ is a square-free-number. Then finding a square in $S_A$ is  the same as finding a number of the form $pz^2$ in $S_B$.

If one replaces squares by higher powers, then the problem becomes easier and asymptotic bounds have been obtained earlier (see next section).

{\it Notations.}   We  use Landau asymptotic  notation such as $O, \Omega, \Theta, o$  throughout the paper, under the assumption that $n \rightarrow \infty$. Notation such as
$\Theta_c(.)$ means that the hidden constant in $\Theta$ depends on a (previously defined) quantity $c$. We will also omit all  unnecessary floors and ceilings. All logarithms have natural base. As usual, $e(x)$ means
$\exp(2 \pi i x) =\cos 2\pi x  + i \sin 2 \pi x $.

\section{The main ideas}\label{section:mainideas}

\noindent The general strategy for attacking  Question \ref{question:1} is as follows.  One first tries to show that if $|A|$ is sufficiently large, then $S_A$ should contain a large additive structure. Next, one would argue that a large additive structure should contain a square.

In previous works \cite{Alon,AlonFreiman,Lipkin,Sarkozy}, the additive structure was a (homogeneous) arithmetic progression.  (An arithmetic progression is homogeneous if it is of the form $\{ld, (l+1)d, \dots, (l+k)d \}$.) It is easy to show that if $P$ is a homogeneous AP of length $C_0 m^{2/3}$ in $[m]$, for some large constant $C_0$, then $P$  contains a square. Notice that the set $S_A$ is a subset of $[m]$ where $m:= |A| n$. Thus, if one can show that $S_A$ contains a homogeneous AP of length $C_0m^{2/3}$, then we are done. S\'ark\"ozy could prove that this is indeed the case, given
$|A| \ge C_1 \sqrt {n \log n}$ for a properly chosen constant $C_1$. This also solves (asymptotically) the problem when squares are replaced by higher powers, since in these cases, the lower bound (which can be obtained by modifying Example \ref{example:1}) is $\Omega (\sqrt n)$.

Unfortunately, $\sqrt n$ is the limit of this argument, since
there are examples of a subset $A$
of $[n]$ of size $\Omega (\sqrt n)$ where the longest AP in $S_A$ is much shorter than $(|A|n)^{2/3}$. In order to present such an example, we will need the following definition (which will play a crucial role in the rest of the paper)

\begin{definition}[Generalized arithmetic
progression-GAP]\label{definition:1}
A generalized arithmetic progression of rank $r$ is a set of the
form

$$Q= \{a_{0} + x_{1} a_{1 } + \dots + x_{r}a_{r} | 0 \le x_{i} \le L_{i}\}. $$

\noindent If all the sums $x_{1}a_{1}+ \dots+ x_{d}a_{d}$ are distinct, we say
that $Q$ is \emph{proper}. If $a_0=0$, we say that $Q$ is
\emph{homogeneous}. (Homogeneous arithmetic progression thus corresponds to the case $r=1$.)  We call $L_1,\dots,L_r$ the sizes of $Q$ and
$a_1,\dots,a_r$ its steps.
\end{definition}

\begin{example}\label{example:2} Consider

$$A:=\{q_1x_1+q_2x_2| 1 \le x_i \le N \} $$

where $q_1 \approx q_2 \approx n^{3/4}$ are different primes and $N = \frac{1}{100} n^{1/4}$. It is easy to show that $A$ is a proper GAP of rank $2$
and  $S_A$ is contained in the  proper GAP

$$ \{q_1x_1+q_2x_2| 1 \le x_i \le 1+\dots +N \}. $$

Thus, the longest AP in $S_A$ has length at most $1+\dots +N =\Theta (n^{1/2})$, while $A$ has cardinality $\Theta (n^{1/2})$.
\end{example}

The key fact that enables us to go below $\sqrt n$ and reach the optimal bound $n^{1/3}$ is a recent theorem of Szemer\'edi and Vu \cite{SzemVu1} that showed that if $|A| \ge C n^{1/3} $ for some sufficiently large constant $C$, then $S_A$ does contain a large proper GAP of rank at most 2.

\begin{lemma} \label{lemma:SV}  \cite{SzemVu1} There are positive  constants $C$ and $c$ such that the following holds. If $A$ is a subset of $[n]$ of cardinality at least
$Cn^{1/3}$, then $S_A$ contains either an AP $Q$ of length $c |A|^2$ or a proper GAP $Q$ of rank 2 and cardinality at least $c |A|^3 $.
\end{lemma}

Ideally, the next step would be showing that a large proper GAP $Q$ (which is a subset of $[|A| n]$) contains a square. Thanks to strong tools from number theory, this is not too hard (though not entirely trivial)  if $Q$ is homogeneous. However, we do not know how to force this assumption.

The assumption of homogeneity is essential, as without this, one can easily run into local obstructions. For example,  if $Q$ is a GAP of the form

$$ \{a_0 + a_1 x_1 + a_2 x_2| 0\le x_i \le L \} $$

where both $a_1$ and $a_2$ are divisible by $6$, but $a_0 \equiv 2 (\bmod 6)$, then
clearly $Q$ cannot contain a square, as $2$ is not a square modulo $6$.

In order to overcome this obstacle, we need to add several twists to the plan. First, we are going to use only a small subset $A'$ of $A$ to create a large GAP $Q$. Assume that $Q$ has the form

$$ \{ a_0 + a_1x_1 + a_2 x_2 | 0 \le x_i \le L \} . $$

($Q$ can also have rank one but that is the simpler case.)
Let $q$ be the g.c.d of $a_1$ and $a_2$. If $a_0$ is a square modulo $q$, then there is no local obstruction and in principle we can treat $Q$ as if it was homogeneous.

In the next move, we  try to add the remaining elements of $A$ (from $A^{''}:= A \backslash A'$) to $a_0$ to make it  a square modulo $q$. This, however, faces another local obstruction. For instance, if in the above example, all elements of $A^{''}$ are divisible by $6$, then $a_0$ will always be $2 (\bmod 6)$ no matter how we add elements from $A^{''}$ to it.

Now comes a key point. A careful analysis  reveals that having all elements of $A^{''}$ divisible by the same integer (larger than one, of course) is the {\it only} obstruction. Thus, we obtain a useful dichotomy: either $S_A$ contains a square or there is an integer $p >1$ which is divisible by all elements of a large subset $A^{''}$ of $A$.

Now we keep working with $A^{''}$. We can write this set as
$\{pb \,\,| b\in B\}$ where $B$ is a subset of $[n/p]$. In order to show that $S_{A^{''}}$ contains a square, it suffices to  show that $S_B$ contains a number of the form $pz^2$. This explains the necessity  of Theorem \ref{theorem:square:p}.

A nice feature of the above plan is that it also works for the more general problem considered  in Theorem \ref{theorem:square:p}. We are going to iterate, setting new $A:= A^{''}$ of the previous step. Since the number of
iterations (i.e., the number of $p$'s) is only $O(\log n)$, if we have
$|A^{''}| \ge (1- \frac{1}{(\log n)^c}) |A|$ in each step, for a sufficiently large constant $c$, then the set $A^{''}$ will never be empty and this guarantees that the process should terminate at some point, yielding the desired result.

In the next lemma, which is the main lemma of the paper,
 we put these arguments into a quantitative form.

\begin{lemma}\label{lemma:main2} The followings holds for any sufficiently large constant $C$.  Let  $p$ be positive
integer less than $n^{2/3}(\log n)^{-C}$  and $A$ be a subset of $ [n/p]$ of cardinality $n^{1/3} (\log n)^C$. Then  there
exists $A'\subset A$ of cardinality $|A'|\le n^{1/3}(\log n)^{C/3}$
such that one of the followings holds (with $A'' :=A\backslash A'$)

\begin{itemize}

\item $S_{A'}$ contains a GAP $$Q=\{r+qx \,\,| 0\le x\le L\}$$ where $L\ge
n^{2/3}(\log n)^{C/4}$ and  $q\le  \frac{n^{2/3}(\log n)^{C/12}}{p}$ and $r\equiv pz^2 (\bmod{q})$ for some integer $z$.

\vskip2mm

\item $S_{A'}$ contains a proper GAP $$Q=\{r+q(q_1x_1+q_2x_2) \,\,| 0\le x_1\le L_1,0\le x_2\le
L_2,(q_1,q_2)=1\}$$  such that $\min(L_1,L_2)\ge n^{1/3}(\log n)^{C/4},
L_1L_2 \ge n (\log n)^{C/2}, q\le \frac{n^{1/3}}{(\log n)^{C/6} p}$ and $r \equiv
pz^2 (\bmod{q})$ for some integer $z$.

\vskip2mm

\item There exists an integer $d>1$ such that $d|a$ for all $a\in
A''$.

\end{itemize}

\end{lemma}

Given this lemma, we can argue as before and show that
 after some iterations, one of the first two cases must occur.
 We show that in these cases the GAP $Q$ should contain a number of the form $pz^2$, using classical tools
from number theory (see Section \ref{section:dim1} and Section \ref{section:dim2}).

\noindent The proof of Lemma \ref{lemma:main2} is technical and requires a preparation involving tools  from both combinatorics and number theory.
These tools will be the focus of the next two sections.

\section{Tools from additive combinatorics}

\noindent This section contains tools from additive combinatorics, which will be useful in the proof of Lemmas \ref{lemma:main1} and \ref{lemma:main2}.
Let $X,Y$ be two sets of numbers. We define

$$X+Y :=\{ x+y \,\,|x\in X, y \in Y \}; X-Y :=\{x-y \,\,| x\in X, y \in Y \}. $$

A translate of a set $X$ is a set $X'$ of the form $X':=\{a+x \,\,| x\in X \}$. For instance, every GAP is a translate of a homogeneous GAP.

The first tool is the so-called Covering lemma, due to Ruzsa
(see \cite{Ruzsa} or  \cite[Lemma 2.14]{TVbook}).

\begin{lemma}[Covering Lemma]\label{lemma:Ruzsa}
Assume that $X,Y$ are finite sets of integers. Then $X$ is
covered by at most $|X+Y|/|Y|$ translates of $Y-Y$.
\end{lemma}

The second tool is  the  powerful inverse theorem of Freiman
\cite{Fre}, \cite[Chapter 5]{TVbook}

\begin{lemma}[Freiman's inverse theorem]\label{theorem:Fre} Let
  $\gamma$ be a given positive number. Let $X$ be a set in $\Z$ such that $|X+X| \le
\gamma |X|$. Then  there exists a proper
GAP $P$ of rank at most $d=d(\gamma)$ and cardinality
$O_{\gamma}(|X|)$ that contains  $X$.
\end{lemma}

Freiman's theorem has the following variant (\cite{Fre, SzemVu1}, \cite[Chapter 5]{TVbook}, which
has a weaker conclusion, but provides the optimal estimate for the rank $d$. This lemma played a key factor in \cite{SzemVu1}.

\begin{lemma}\label{theorem:Fre1} Let
  $\gamma < 2^d$ be a given positive number. Let $X$ be a set in $\Z$ such that $|X+X| \le \gamma |X|$. Then  there exists a proper
GAP $P$ of rank at most $d$ and cardinality
$O_{\gamma}(|X|)$ that contains  $X$.
\end{lemma}

This lemma will not be sufficient for our purpose here. We are going to need
the following refinement, which can be proved by combining Lemma \ref{theorem:Fre1} and the Covering lemma.

\begin{lemma}\label{lemma:GreenTao} \cite{GT} \cite[Chapter 5]{TVbook} Let
  $\gamma, \delta$ be positive constants. Let $X$ be a set in $\Z$ such that $|X+X| \le
\gamma |X|$. Then  there exists a proper
GAP $P$ of rank at most $\lfloor \log_2\gamma+\delta \rfloor$ and cardinality
$O_{\gamma,\delta}(|X|)$ such that $X$ is covered by
$O_{\gamma,\delta}(1)$ translates of P.
\end{lemma}

\noindent We say that a GAP $Q = \{a_{0} + x_{1} a_{1 } + \dots x_{d}a_{d} | 0
\le x_{i} \le L_{i}\}$ is {\it positive} if its steps $a_i$'s
are positive. A useful observation is that if the  elements of $Q$ are
positive, then $Q$ itself can be brought into a positive form.

\begin{lemma} \label{lemma:positive}
A   GAP with positive elements can be
brought into a  positive form.
\end{lemma}

\begin{proof} (Proof of Lemma \ref{lemma:positive}.)  Assume that

$$Q= \{a_{0} + x_{1} a_{1 } + \dots x_{d}a_{d} | 0 \le x_{i} \le L_{i}\}. $$

By setting $x_{i}=0$, we can conclude that $a_{0} >0$. Without loss
of generality, assume that $a_{1}, \dots, a_{j } <0$ and $a_{j+1},
\dots, a_{d} >0$. By setting $x_{i}=0$ for all $i >j$ and
$x_{i}=L_{i}, i \le j$, we have

$$a'_{0} := a_{0} + a_{1}L_{1} + \dots a_{j} L_{j} >0. $$

Now we can rewrite $Q$ as

$$Q:= \{a_{0}' + x_{1}(-a_{1}) + \dots+ x_{j} (-a_{j}) + x_{j+1} a_{j+1} + \dots x_{d}a_{d}| 0 \le x_{i } \le L_{i}\},
$$

\noindent completing the proof.\end{proof}

\noindent Since we only deal with positive integers, this lemma allows us to  assume that all GAPs  arising in the proof are  in positive
form.

  Using the above tools and ideas from \cite{SzemVu1}, we will prove Lemma \ref{lemma:main1} below, which asserts that if a  set $A$ of $[n/p]$ is sufficiently dense, then there
exists a small set $A'\subset A$ whose subset sums contain a large
GAP $Q$ of small rank. Furthermore, the set $A''=A\backslash
A'$ is contained in only a few translates of $Q$. This lemma will serve as a base from which we will attack Lemma \ref{lemma:main2}, using number theoretical tools discussed in the next section.

\begin{lemma}\label{lemma:main1} The following holds for all sufficiently large constant  $C$. Let $p$ be positive
integer less than $n^{2/3}(\log n)^{-C}$ and  $A$ be a subset of $ [n/p]$ of cardinality $n^{1/3}(\log n)^C$. Then there
exists a subset $A'$ of $ A$ of cardinality $|A'|\le n^{1/3}(\log n)^{C/3}$
such that one of the followings holds (with $A'':=A\backslash A'$):

\begin{itemize}

\item $S_{A'}$ contains an AP $$Q=\{r+qx \,\,| 0\le x\le L\}$$ where  $L\ge
n^{2/3}(\log n)^{C/2}$ and there exist $m=O(1)$ different numbers
$s_1,\dots,s_m$ such that $A''\subset \{s_1,\dots,s_m\}+Q$.

\item $S_{A'}$ contains a proper GAP $$Q=\{r+a_1x_1+a_2x_2) \,\,| 0\le x_1\le L_1,0\le x_2\le
L_2$$ such that $L_1L_2 \ge n (\log n)^{C/2} \}$ and there
exists $m=O(1)$ numbers $s_1,\dots,s_m$ such that $A''\subset
\{s_1,\dots,s_m\}+Q$.

\end{itemize}

\end{lemma}

{\it Remark.} The proof actually gives a better lower bounds for $L_1L_2$ in the second case ($2C/3$ instead of $C/2$), but this is not important in applications.

\section{Tools from number theory}

{\it Fourier Transform and Poisson summation.} Let $f$ be a function with support on $\Z$. The Fourier transform $\widehat {f}$ is defined as

$$\widehat {f} (w):= \int_{\R} f(t) e(-wt) \, \, dt . $$

The classical Poisson summation formula  asserts that

\begin{equation} \label{PSF}
\sum_{n= -\infty}^{\infty} f(t + nT) = \frac{1}{T}
\sum_{m=-\infty}^{\infty} \widehat{f} (\frac{2\pi m}{T}) e(mt/T).
\end{equation}

For more details, we refer to  \cite[Section 4.3]{IKbook}.

{\it \noindent Smooth indicator functions.} We will use the following well-known construction (see for instance \cite[Theorem
18]{Green} for details).

\begin{lemma}\label{lemma:Fourier} Let $\delta <1/16$ be a positive
constant and let $[M,M+N]$ be an interval. Then there exists a real function $f$ satisfying the
following

\begin{itemize}
\item $0\le f(x) \le 1$ for any $x\in \R.$
\item $f(x)=0$ if $x\le M$ or $x\ge M+N.$
\item $f(x)=1$ if $M+\delta N \le x \le M+ N(1-\delta).$
\item $|\widehat{f}(\lambda)| \le 16 \widehat{f}(0) \exp(-\delta |\lambda N|^{1/2})$ for every $\lambda$.
\end{itemize}
\end{lemma}

{\it \noindent  A Weyl type estimate.} Next, we need a Weyl type estimate for exponential sums.

\begin{lemma}\label{lemma:Weyl} For any positive constant $\ep$
  there exist positive constants $\alpha= \alpha(\ep)$ and $c(\ep)$
such that the following holds. Let $a,q$ be co-prime integers, $\theta$ be a real number, and $I$ be an interval of length $N$. Let $M$ be a positive number such that  $MN\ge q^{1+\ep}$. Then,

$$ \sum_{\scriptstyle |m| \le M \atop \scriptstyle m\neq 0}|\sum_{z\in I}e(\frac{amz^2}{q}+\theta m
z)| \le c (M\sqrt{N}+ \frac{MN}{\sqrt{q}} ) (\log MN)^{\alpha}.$$
\end{lemma}

{\it \noindent Quadratic residues.} Finally, and most relevant to our problem, we need the following lemma, which shows the existence of
integer solutions with given constrains for a quadratic equation.

\begin{lemma}\label{lemma:main3}
There is an absolute constants $D$ such that the following holds.  Let
$a_1,\dots,a_d,r, p,q$ be integers such that  $p,q>0$ and  $(a_1,\dots,a_d,q)=1$. Then the
equation
\begin{equation}\label{equation:lemma3:1}
a_1x_1+\dots+a_dx_d + r \equiv pz^2 (\bmod{q})
\end{equation}

has an integer solution $(z, x_{1}, \dots x_{d})$ satisfying $0
\le x_i \le (pq)^{1/2}(\log q)^{D} $.
\end{lemma}

The rest of the paper is organized as follows.  The proof of the
combinatorial statement, Lemma \ref{lemma:main1}, comes first in Section \ref{section:lemma1}. We then start the number theoretical part by
giving a proof for Lemma \ref{lemma:Weyl}. The verification of Lemma
\ref{lemma:main3} comes in Section \ref{section:lemma3}.
 After all these preparations, we will be able to establish Lemma \ref{lemma:main2} in Section \ref{section:lemma2}.
 The proof of the main result,  Theorem
\ref{theorem:square:p}, is presented  in Sections \ref{section:dim1} and \ref{section:dim2}.

\section{Proof of Lemma \ref{lemma:main1}}\label{section:lemma1}

\noindent We repeat some arguments from  \cite{SzemVu1} with certain  modifications. The extra information we want to get here (compared with
what have already been done \cite{SzemVu1}) is the fact that the set $A^{''}$ is covered by only few translates of $Q$.

\subsection{An algorithm} Let $A'$ be a
subset of cardinality $|A'|=n^{1/3}(\log n)^{C/3}$ and let
$A'' :=A\backslash A'$. By a simple combinatorial argument  (see \cite[Lemma 7.9]{SzemVu1}), we can find
in $A'$ disjoint subsets $A'_1,\dots,A'_{m_1}$ such that
$|A_i'|\le 20\log_2 |A'|$ and $|l_1^\ast {A_i'}|\ge |A'|/2$ where

\begin{equation}\label{equation:lemma1:l1,m_1}
l_1 \le 10 \log_2 |A'| \mbox{ and } m_1 = |A'|/(40 \log_2 |A'|).
\end{equation}

\noindent (For the definition of $l^*A$ see the beginning of the introduction.)

\noindent Without loss of generality, we can
assume that $m_1$ is a power of 4. Let
$B_1,\dots,B_{m_1}$ be subsets of cardinality $b_1=|A'|/2$ of the
sets $l_1^\ast A_1',\dots, l_1^\ast A_{m_1}'$ respectively. Following \cite[Lemma 7.6]{SzemVu1}), we will run an algorithm with the $B_i$'s as input.
The goal of this algorithm is to produce a GAP which has nice relations with $A^{''}$ (while still not as good as the GAP we wanted in the lemma).  In the next few paragraphs, we are going to describe this algorithm.

\noindent At the first step, set $B_1^1:=B_1,\dots,B_{m_1}^1:=B_{m_1}$ and let
${\mathfrak{B^1}}=\{B_1^1,\dots,B_{m_1}^1\}$. Let $h$ be a large constant to be determined later.

\noindent  At the $(t+1)$-th step, we choose indices $i,j$ and elements
$a_1,\dots,a_h\in A''$ that maximizes the cardinality of $\cup_{d=1}^h (B_i^t+B_j^t+a_d)$ (if there are many choices, choose one arbitrarily). Define ${B^{t+1}_1}'$ to be the union. Delete
from $A''$  the used elements $a_1,\dots,a_h$, and remove from
$\mathfrak{B^t}$ the used sets $B_i^t,B_j^t$. Find the next maximum union $\cup_{k=1}^h
B_i^t+B_j^t+a_k$ with respect to the updated sets $\mathfrak{B^t}$ and
$A''$.

\noindent Assume that we have created $m_{t+1} :=m_t/4$ sets
${B_1^{t+1}}',\dots, {B_{m_{t+1}}^{t+1}}'$. By the algorithm, we have

$$ |{B_1^{t+1}}'| \ge \dots \ge |{B_{m_{t+1}}^{t+1}}'|:= b_{t+1}.$$

\noindent Now for each $1\le i \le m_{t+1}$ we choose a subset $B_i^{t+1}$ of
cardinality exactly $b_{t+1}$ in ${B_i^{t+1}}'$. These $m_{t+1}$ sets (of the same cardinality)  from a collection
$\mathfrak{B^{t+1}}$, which is the output of the $(t+1)$-th step.

\noindent Since $m_{t+1} =m_t/4$,
there are still $m_{t}/2$ unused sets $B_i^t$ left in
$\mathfrak{B^t}$. Without loss of generality, assume that those
are $B_1^{t},\dots,B_{m_t/2}^t$. With a slight abuse of notation, we use $A^{''}$ at every step, although this set loses a few elements each time. (The number of deleted elements is very small compared to the size of $A^{''}$.)

\noindent  Let $l_{t+1} :=2l_t+1$. Observe that

\begin{itemize}

\item $l_t \le 2^t l_1$ (by definition);

\vskip .1in

\item $b_t \le l_t n/p$ (since $\cup_{d=1}^h (B_i^{t-1}+B_j^{t-1}+a_d) \subset
[l_tn/p]$);

\vskip .1in

\item \begin{equation} \label{eqn:BB} |\cup_{d=1}^h B_i^t+B_j^t+a_d|\le b_{t+1}\end{equation} for all $1\le i<j
  \le m_t/2$ and $a_1,\dots,a_h\in A''$ (by the algorithm, as it always chooses a union with maximum size).

\end{itemize}

\noindent Now let $c$ be a large constant and $k$ be the largest index such that  $b_{i}\ge
cb_{i-1}$ for all $i\le k$. Then we have

$$c^k b_1 \le b_k \le l_k n/p.$$

\noindent Since $b_1=|A'|/2$ and $l_k\le 2^k
l_1$, we deduce an upper bound for $k$,

$$k \le \log_{c/2} \frac{l_1n}{b_1 p}.$$

\noindent Next, by the definition of $k$, we have $b_{k+1}\le cb_{k}$. By (\ref{eqn:BB}),
the following holds for all unused sets $B_i^{k},B_j^{k}$ (with
$1\le i \le j\le m_{k}/2$) and for all $a_1,\dots,a_h \in A''$:

$$|\cup_{d=1}^h (B_i^{k}+B_j^{k}+a_d)| \le b_{k+1} \le c b_{k} = c|B_i^{k}|.$$

\noindent In particular

$$|B_1^{k}+B_i^{k}| \le c|B_1^{k}|$$

\noindent holds for all $2\le i \le m_k/2$.

\noindent By  Plunnecke-Ruzsa estimate (see \cite[Corollary 6.28]{TVbook}),  we have

$$|B_1^k+B_1^k|\le c^2 |B_1^k|.$$

\noindent It then follows from Freiman's theorem, Theorem \ref{theorem:Fre}, that there exists a proper GAP $R$ of rank
$O_c(1)$, of size $O_c(1)|B_1^k|$ such that $R$ contains $B_1^{k}$.
Furthermore, by Lemma \ref{lemma:Ruzsa}, $B_i^k$ is contained in $c$
translates of $B_1^k-B_1^k$, thus $B_i^k$ is also contained in $O_c(1)$
translates of $R$.

\noindent Before continuing, we would like to point out that the parameter $h$  has not yet played any role in the arguments. The freedom of choosing $h$ will be important in what follows. We are going to obtain the desired GAP $Q$ (claimed in the lemma) from $R$ by a few additional operations.

\subsection{Creation of many similar GAPs.} One problem with $R$ is that its cardinality can be significantly smaller than the bounds on $Q$ in Lemma \ref{lemma:main1}. We want to obtain larger GAPs by adding many translates of $R$. While we cannot do exactly this, we can do nearly as good by the following argument, which creates many GAPs which are translates of each other and have cardinalities comparable to that of $R$.

By the pigeon hole principle, for $i\le m_k/2$,  we can
 find a set $B_i' \subset B_i^{k}$ with cardinality  $\Theta_c(1)b_{k}$
 which is  contained in one translate of
$R$.

\noindent By \cite[Lemma 5.5]{SzemVu1}, there exists $g=O_{c}(1)$ such that
$B_1'+\dots+B_g'$ contains a proper GAP $Q_1$ of cardinality
$\Theta_c(1)|R|$. Create $Q_2$ by summing $B_{g+1}',\dots, B_{2g}'$, and so on. At the end we
obtain $\frac{m_{k}}{2g} = \Theta_c(1)m_{k}$ such GAPs. Following \cite[Lemma 5.5]{SzemVu1}, we can require the $Q_i$'s to have the properties below

\begin{itemize}

\item rank($Q_i)=$rank($R)=O_{c}(1)$;

\vskip .1in

\item $|Q_i|=\Theta_{c}(1)|R|=\Theta_c(1)b_k$;

\vskip .1in

\item each $Q_i$ is a subset of a translate of $gR$. Thus by Lemma \ref{lemma:Ruzsa}, $R$ is contained in $O_c(1)$
translates of $Q_i-Q_i$;

\vskip .1in

\item the $j$-th size of $Q_i$ is different from $j$-th size of $R$ by
a (multiplicative) factor of order  $\Theta_c(1)$, for all $j$;

\vskip .1in

\item the $j$-th step of $Q_i$ is a multiple of the $j$-th step of $R$ for all $j$;

\end{itemize}

\noindent Thus, by the pigeon hole principle and truncation (if necessary) we can obtain $m'=\Theta_c(m_{k})$
GAPs, say, $Q_1,\dots, Q_{m'}$, which
are translate of each other.  An important remark here is that since the $Q_i$ are obtained from summing different $B$'s, the sum
$Q_1+ \dots+ Q_{m'}$ is a subset of $S_{A'}$. The desired GAP $Q$ will be a subset of this sum.

\subsection{Embedding  $A''$} In this step, we embed $A^{''}$ in a union of few translates of a GAP $Q_1$ of constant rank.

We set the (so far untouched) parameter $h$ to be sufficiently large so that

$$ \Theta_c(1)=h  >  c|B_1^{k}|/|B_1'|.$$

\noindent Let $d$ be the largest number such that there are $d$ elements
$a_1,\dots,a_d$ of $A''$ for which the sets  $B_1'+B_2'+a_i$ are
disjoint. Assume for the moment that $d\ge h$, then we would have

$$|\cup_{i=1}^h (B_1'+B_2'+a_i)= h|B_1'+B_2'| \ge h|B_1'| > c |B_1^{k}|$$

\noindent However, this is impossible because $\cup_{i=1}^h
(B_1'+B_2'+a_i)\subset \cup_{i=1}^h (B_1^k+B_2^k+a_i)$ and the
latter has cardinality less than $c|B_1^k|$ by definition. Thus we have $d <
h$. So $d=O_c(1)$.

\noindent  Let us fix $d$ elements $a_1,\dots,a_d$ from $A''$ which attained the disjointness in the definition of $d$. By the maximality of $d$, for any $a\in A''$ there exists
$a_i$ so that $(B_1'+B_2'+a) \cap (B_1'+B_2'+a_i) \neq \emptyset$.
Hence

$$a-a_i \in B^k_1+B^k_2-(B^k_1+B^k_2)=(B^k_1-B^k_1)+(B^k_2-B^k_2)\subset
2R-2R.$$

\noindent Thus $A''$ is covered by at most $d=O_c(1)$ translates
of $2R-2R$. On the other hand, since $R$ is contained in
$O_c(1)$ translates of $Q_1-Q_1$, $2R-2R$ is contained in
$O_c(1)$ translates of $4Q_1-4Q_1$. It follows that
that $A''$ is covered by $O_c(1)$ translates of $Q_1$.

The remaining problem here is that $Q_1$ does not yet have the required rank and cardinality. We will obtain these by adding the $Q_i$ together (recall that these GAPs are translates of each other) and using a rank reduction argument, following \cite{SzemVu1} (see also \cite[Chapter12]{TVbook}).

\subsection{Rank reduction}
Let $P$ be the homogenous translate of $Q_1$( and also  of
$Q_2,\dots,Q_{m'})$. Recall that

$$|P|=|Q_1| =\Theta_c(b_k)=\Omega_c({c^kb_1}).$$

\noindent and also $$m'=\Theta_c(m_k)=\Theta_c(\frac{b_1}{4^k}),
\mbox{ and } l_{k+1}\le 2^{k+1}l_1 .$$

\noindent Set $l: =\min\{m', |A'|/2l_{k+1}\}$. Recall that
 $|A'| = n^{1/3} (\log n)^{C/2}$, $l_1 \le 10 \log_2 |A'|$ and $b_1 =|A'|/2$. By choosing $c$ and $C$ sufficiently large, we can guarantee that

\begin{equation}\label{equation:lemma1:largeenergy}
l|P|\ge n^{2/3}(\log n)^{C/2} \,\,\, ; l^2|P|\ge n (\log n)^{2C/3}.
\end{equation}

\noindent and also

\begin{equation}\label{equation:lemma1:toolargeenergy}
l^3|P|\ge n^{4/3}(\log n)^C
\end{equation}

\noindent Now we invoke Lemma
\ref{lemma:GreenTao} to find a large GAP in $lP$. Assume, without loss of generality, that $l=2^s$ for some integer $s$.
We start with $P_0 :=P$ and $\ell_0 :=l$. If $2^sP_0$ is  proper,
then we stop. If not, then there exists a smallest index $i_1$ such that
$2^{i_1}P_0$ is proper but $2^{i_1+1}P_0$ is not.

\noindent By Lemma \ref{lemma:GreenTao} (applying to $2^{i_1}P_0$; see also \cite[Lemma
4.2]{SzemVu1}) we can find a  GAP $S$ which
contains a $\Theta_c(1)$ portion  of $2^{i_1}P_0$ such that
$rank(S)< r :=rank (2^{i_1} P_0)$. We denote by $P'$ the intersection of
$S$ and $2^{i_1} P_0$.

\noindent By \cite[Lemma 5.5]{SzemVu1}, there is a constant
$g=\Theta_c(1)$ such that the set $2^gP'$ contains a proper GAP $P_{1}$ of
rank equals $rank S$ and  cardinality $\Theta_c(1)|2^{i_1}P_0|$.
 Set $\ell_1:=\ell_0/2^{i_1+g}$ if $\ell_0/2^{i_1+g}\ge 1$ and proceed with $P_1, \ell_1$ and so on. Otherwise we  stop.

\noindent Observe that if $2^{i_j}P_j$ is proper, then
$|2^{i_j}P_j|=(1+o(1))2^{i_j r_j}|P_j|$, where $r_j$ is the rank of $P_{j}$.

\noindent As the rank of $P_0$ is $O_c(1)$, and $r_{j+1}\le r_j-1$, we  must stop after $\Theta_c(1)$
steps. Let $Q'$ be the  proper GAP $Q'$ obtained when we stop. It has rank $d'$, for some integer $d' < r$ and cardinality at
least $\Theta_c(1)\ell_0^{d'}|P_0|=\Theta_c(1)l^{d'}|P|$. On the other
hand, since a translate of $lP$ is contained in $S_{A'}$, $|Q'|\le
|A'|n/p \le |A'|n $, that is $\Theta_c(1)l^{d'}|P|\le |A'|n$. Because of
(\ref{equation:lemma1:toolargeenergy}), this holds only if $d'\le 2$.

\subsection{Properties of $Q$.}
 We showed that $A''$ is contained in $\Theta_c(1)$ translates of
$Q_1$, thus it is contained in $\Theta_c(1)$ translates of $2^{i_1}P_0$.

\noindent By Lemma \ref{lemma:GreenTao}, $2^{i_1}P_0$ is covered
by $\Theta_c(1)$ translates of $S$. It follows that
$A''$ is  contained in $\Theta_c(1)$ translates of $S$. On the other hand, by Lemma
\ref{lemma:Ruzsa}, $S$ is contained in $\Theta_c(1)$ translates of
$P'-P'$. Thus $A''$ is contained in $\Theta_c(1)$ translates of
$P'-P'$, and hence in $\Theta_c(1)$ translates of $(2^gP')-(2^gP')$
as well. Furthermore, by Lemma \ref{lemma:Ruzsa}, $2^gP'$ is covered
by $O_c(1)$ translates of $P_1$ (more precisely, $P_1-P_1$), thus we conclude that $A''$ is covered by
$\Theta_c(1)$ translates of $P_1$. Because we stop after
$\Theta_c(1)$ steps, a similar relation is also valid between $A''$
and any $P_j$. Thus, $A^{''}$ is covered by $O_c(1)$ translates of $Q'$.

\noindent Furthermore,  $Q'$ is a subset of $lP$. Thus a translate $Q$ of  $Q'$ lies in $Q_1+\dots+Q_{m'}\subset S_{A'}$. This $Q$ has rank
$1\le d'\le 2$ and cardinality $|Q| =|Q'| \ge \Theta(1)l^{d'}|A'|$.
(The right hand side satisfies the lower bounds claimed in Lemma \
\ref{lemma:main1}, thanks to (\ref{equation:lemma1:largeenergy}).)
This is the GAP claimed in Lemma \ref{lemma:main1} and our proof is complete.

\section{Proof of Lemma \ref{lemma:Weyl}}\label{section:lemma:Weyl}

If $q$ is a prime, the lemma is a corollary of the well known Weyl's estimate
(see  \cite{IKbook}. We need to add a few arguments to handle the general case. The following lemma will be useful.

\begin{lemma}\label{lemma:divisor}
Let $\tau(n)$ be the number of positive divisors of $n$.
For any given $k\ge 3$ there exists a positive
constant $\beta(k)$ such that the following
holds for every $n$.

$$\tau(n)= O_k( \sum_{\scriptstyle d|n \atop \scriptstyle d\le
n^{1/k}}\tau(d)^{\beta(k)}).$$

\end{lemma}
\begin{proof}(Proof of Lemma \ref{lemma:divisor}). We can set
 $\beta(k)=k\log (k+1)$. We factorize $n$ in the following specific way

$$n=\prod_{i=1}^u p_i^{a_i}\prod_{j=1}^vq_j^{b_j}$$

where $p_1\le\dots \le p_u,\,\, q_1\le \dots \le q_v$ are primes and $a_i\ge k> b_j\ge 1$. Set

$$d:=\prod_{i=1}^u p_i^{\lfloor \frac{a_i}{k}\rfloor}\prod_{j\le
\lfloor \frac{v}{k}\rfloor}q_j.$$

\noindent Then $d\le n^{1/k}$ by definition and

$$(k+1)^k \tau(d)^{\beta(k)}=(k+1)^k 2^{\lfloor
\frac{v}{k}\rfloor k\log (k+1)}\prod_{i=1}^u(\lfloor
\frac{a_i}{k}\rfloor+1)^{k\log (k+1)} \ge (k+1)^{v}\prod_{i=1}^u(1+a_i)\ge \tau(n),$$ completing the proof.\end{proof}

\noindent Now we start the proof of Lemma \ref{lemma:Weyl}.
Let $S:= \sum_{\scriptstyle |m| \le M \atop \scriptstyle m\neq 0}|\sum_{z\in I}e(\frac{amz^2}{q}+\theta m
z)|$. Following Weyl's argument, we use
 Cauchy-Schwarz and the triangle inequality to obtain

$$S^2 \le 2M\sum_{{\scriptstyle |m| \le M \atop \scriptstyle m\neq 0}} \sum_{z_1,z_2\in I}
e(\frac{am(z_1-z_2)(z_1+z_2)}{q}+\theta m(z_1-z_2)).
$$

\noindent For convenience, we
 change the  variables, setting  $u:= z_1-z_2,v:= z_2$, then

$$S^2 \le 2M\sum_{{\scriptstyle |m| \le M \atop \scriptstyle m\neq 0}} \sum_{|u|\le N}e(\frac{amu^2}{q}+\theta mu)\sum_{v\in I, v\in I-u}
e(\frac{2amuv}{q})$$

$$\le 2M \sum_{{\scriptstyle |m| \le M \atop \scriptstyle m\neq 0}} \sum_{|u|\le N}|\sum_{v\in I, v\in I-u}
e(\frac{2amuv}{q})|.$$

\noindent Next, using the basic estimate (see \cite[Section 8.2]{IKbook}, for instance)

$$| \sum_{K_0< k \le K_0+K}e(\omega k) |\le \min(K,\frac{1}{\|2\omega \|})$$

\noindent we obtain that

$$S^2 \le 2M \sum_{{\scriptstyle |m| \le M \atop \scriptstyle m\neq 0}} \sum_{|u|\le N}\min (N,
\frac{1}{\|2amu/q\|}).$$

\noindent To estimate the right hand side, let $N_r$ be the number of pairs $(m,u)$ such that $2amu \equiv
r(\bmod{q})$. (In what follows, it is useful to keep in mind that $a$ and $q$ are co-primes.) We have

\begin{equation}\label{equation:lemma:Weyl:1}
S(M,N,q)^2 \le 2M \left (N_0N + \sum_{1\le r \le
q/2}(N_r+N_{q-r})\frac{q}{r}\right ).
\end{equation}

To finish the proof, we are going to derive a (uniform) bound for the $N_r$'s. For $0\le r \le q-1$ let $0\le r_a\le q-1$ be the only
number such that $ar_a \equiv r(\bmod {q})$. Thus $2amu \equiv
r(\bmod{q})$ is equivalent with $2mu \equiv r_a(\bmod {q})$.

\noindent First we consider the case $r\neq 0$, thus $r_a\neq 0$. Write $2mu=r_a+sq$. It
is clear that $r_a+sq\neq 0$ for all $s$. Since $2mu\le 2MN$, we have $|s|\le 2MN/q$. For each
given $s$ the number of such pairs $(m,u)$ is bounded by
$\tau(r_a+sq)$.

\noindent Choose $k=\max(\frac{1}{\ep}+2,3)$, then $MN/q\ge (MN)^{2/k}$ by the assumption $MN\ge q^{1+\ep}$. It follows from
Lemma \ref{lemma:divisor} that, for $r\neq 0$,

\begin{align*} N_r \le \sum_{|s|\le 2MN/q}\tau(r_a+sq)
 &= O_{\ep} (\sum_{d\le (MN)^{1/k}}\tau(d)^{\beta(k)}(\sum_{\scriptstyle |s|\le
4MN/q \atop \scriptstyle d|r_a+sq}1)\\
& =O_{\ep} ( \sum_{d\le (MN)^{1/k}}\tau(d)^{\beta(k)}(\frac{4MN}{qd} +O(1))) \\
&= O_{\ep} (\frac{MN}{q} \sum_{d\le (MN)^{1/k}} \frac{\tau(d)^{\beta(k)}}{d} + O((MN)^{1/k+o(1)})) \\
&= O_{\ep }(\frac{MN}{q} \sum_{d\le (MN)^{1/k}} \frac{\tau(d)^{\beta(k)}}{d}). \end{align*}

\noindent Notice that $\sum_{d\le x}\tau(d)^{\beta(k)}\ll
x\log^{\beta'(k)} x$ for some positive constant $\beta'(k)$ depending
on $\beta(k)$ (see \cite[Section 1.6]{IKbook}, for instance). By summation by parts we deduce that

$$N_r =O_{\ep}( \frac{MN}{q} \log^{\beta''(k)} (MN))$$

\noindent for some positive constant $\beta''(k)$ depending on $\beta'(k)$.

\noindent Now we consider the case $r=0$. The equation $2mu=sq$ has at most $\tau(sq)$
solution pairs $(m,u)$, except when $s=0$, the case that has $2M$
solutions $\{(m,0);|m|\le 2M, m\neq 0\}$. Thus we have

$$N_0 \le 2M+  \sum_{|s|\le 2MN/q,s\neq 0}\tau(sq),$$

\noindent and hence,

$$N_0 =  O_{\ep}( 2M+\frac{MN}{q} \log^{\beta''(k)} (MN)).$$

\noindent Combining  these estimates with
(\ref{equation:lemma:Weyl:1}), we can conclude that

$$S(M,N,q) \ll_{\ep}
(M\sqrt{N}+MN/\sqrt{q})\log^{\alpha}(MN)$$

\noindent for some sufficiently large constant $\alpha=\alpha(\ep)$.

\section{Proof of Lemma \ref{lemma:main3}}\label{section:lemma3}

We are going to need the following simple fact.

\begin{fact}\label{fact:1} Let $a_1,\dots,a_m,q$ be integers
such that $(a_1,\dots,a_m,q)=1$. Then we can select a decomposition
$q=q_1\dots q_l$ of $q$ and $l$ different numbers $a_{i_1},\dots,a_{i_l}$ of $\{a_1,\dots,a_m\}$ (for some $l \ge 1$) such that

$$(q_i,q_j)=1
\mbox{ for evey $i \neq j$ and } (a_{i_j},q_j)=1 \mbox{ for every $j$}.$$
\end{fact}

\begin{proof}(of Fact \ref{fact:1})
Let $q=q_1'\dots q_k'$ be the decomposition of $q$ into prime
powers. For each $q_i'$ we assign a number $a_i'$ from $\{a_1, \dots, a_m \}$ such that $(q_i',a_i')=1$ (the same $a_i$ may be assigned to many $q_j'$). Let $a_{i_j}$'s be the collection of the $a_i'$'s
without multiplicity. Set  $q_{j}$ to be the product of all
$q_i'$  assigned to $a_{i_j}$.
\end{proof}

\noindent The core of the proof
of Lemma \ref{lemma:main3} will be the following proposition, which is basically the  case of one variable in a slightly more general setting.

\begin{proposition}\label{proposition:square:local:1}
There is a constants $D$ such that the following holds.
For given integers $g,h,p,t,z_1; g,h,p >0$ there exist integers $x\in [0,(ph)^{1/2} (\log h)^{D}]$ and $z_2$ such that $gx+pz_1^2+tk \equiv pz_2^2 (\bmod h)$,
where $k=(g,h)$ .
\end{proposition}

\noindent Lemma \ref{lemma:main3} follows from Fact \ref{fact:1} and Proposition
\ref{proposition:square:local:1} by an inductive argument. Indeed, by
the above fact we may assume that $q=q_1\dots q_l$ where
$(a_i,q_i)=1$, and so

$$(a_l,q)|q_1\dots q_{l-1}.$$

\noindent Now if Lemma \ref{lemma:main3} is
true for $l-1$ variables, i.e. there are appropriate
$x_1,\dots,x_{l-1}$ such that $a_1x_1+\dots a_{l-1}x_{l-1} + r =
pz_1^2+tq_1\dots q_{l-1}$. Then we apply Proposition
\ref{proposition:square:local:1} for $q=h,g=a_l$ to find $x_l$. It thus
remains to justify Proposition \ref{proposition:square:local:1}.

\begin{proof}(of Proposition \ref{proposition:square:local:1})
Without loss of generality we assume that $h\ge 3$. As $k=(g,h)$, we can write $g=ka,h=kq$ where $(a,q)=1$. We shall find a solution in the form $z_2=z_1+zk$. Plugging in $z_2$ in this form and simplifying by $k$, we end up with the equation

$$ax+t \equiv pkz^2 +2pz_1z(\bmod{q}).$$ or equivalently,

\begin{equation}\label{equation:lemma3:Proposition:1}
x \equiv \bar{a}pkz^2 +2\bar{a}pz_1z-\bar{a}t (\bmod q) \end{equation}

 \noindent where $\bar{a}$ is the reciprocal of $a$ modulo $q$,
$a\bar{a}\equiv 1 (\bmod q)$.

\noindent Our task is to find $x\in[0,(ph)^{1/2}(\log h)^D]$ such that
(\ref{equation:lemma3:Proposition:1}) holds for some integer
$z$. Notice that if $q$ is small and $D$ is large then
$(ph)^{1/2}(\log h)^D \ge (\log 3)^D$, therefore the interval
$[0,(ph)^{1/2}]$ contains every residue class modulo $q$; as a result,
(\ref{equation:lemma3:Proposition:1}) holds trivially. From now on we
can assume that $q$ is large,

\begin{equation}\label{assumeq}
q\ge \exp\big(16(6(\alpha+1)/e)^{\alpha+1}\big)
\end{equation}

where $c,\alpha$ are constants arising from  Lemma \ref{lemma:Weyl}
with $\ep = 1/3$.

\noindent Let $s=(pk, q)$; so we can write $pk=sp',q=sq'$ with
$(p',q')=1$.

\noindent Let $D$ be a large constant (to be determined later) and set

$$L:=(sq)^{1/2}(\log q)^{D}/2 \mbox{ and } I:=[L,2L].$$

 Note that

$$ph=pkq=sp'q\ge sq.$$

\noindent Thus we have

$$I\subset [0,(ph)^{1/2}(\log h)^{D}].$$

\noindent Let $f$ be a smooth function defined with respect to the interval $I$ (as in Lemma
\ref{lemma:Fourier}). For fixed $z\in [1,q]$ the numbers of $x$ in
$[0,(sq)^{1/2}\log^D q]$ satisfying
(\ref{equation:lemma3:Proposition:1}) is at least

$$N_z :=\sum_{m\in \Z}f(\bar{a}pkz^2 +2\bar{a}pz_1z-\bar{a}t+mq).$$

By Poisson summation formula (\ref{PSF})

$$N_z = \sum_{m\in \Z} \frac{1}{q}\widehat{f}(\frac{m}{q}) e(\frac{(\bar{a}pkz^2 +2\bar{a}pz_1z-\bar{a}t)m}{q}).$$

\noindent By summing over $z \in [1,q]$ we obtain

$$N:=\sum_{z=1}^q N_z = \frac{1}{q}\sum_{m\in \Z}\widehat{f}(\frac{m}{q})\sum_{z=1}^{q}e(\frac{(\bar{a}pkz^2 +2\bar{a}pz_1z-\bar{a}t)m}{q}).$$

\noindent To conclude the proof, it suffices to show that $N >0$.
We are going to show (as fairly standard in this area) that the sum is dominated by the contribution of the zero term.

By the triangle inequality, we have

$$|N-\widehat{f}(0)| \le \frac{1}{q}\sum_{m\in \Z,m\neq 0}|\widehat{f}(\frac{m}{q})||\sum_{z=1}^{q}e(\frac{(\bar{a}pkz^2 +2\bar{a}pz_1z)m}{q})|.$$

Let $\gamma_1,\gamma_2$ be a sufficiently large constant and let

$$L':= \frac{\gamma_1 q (\log q)^{\gamma_2} } {L} . $$

\noindent Set

$$S_1 :=\frac{1}{q}\sum_{|m|\ge L'}|\widehat{f}(\frac{m}{q})||\sum_{z=1}^{q}e(\frac{(\bar{a}pkz^2 +2\bar{a}pz_1z)m}{q})|$$

and

$$S_2 :=\frac{1}{q}\sum_{\scriptstyle |m|\le L' \atop \scriptstyle m\neq 0}|\widehat{f}(\frac{m}{q})||\sum_{z=1}^{q}e(\frac{(\bar{a}pkz^2 +2\bar{a}pz_1z)m}{q})|.$$

\noindent We then have

$$|N-\widehat{f}(0)| \le S_1+S_2.$$

In what follows, we show that both $S_1$ and $S_2$ are less than
$\widehat{f}(0)/4$.

{\it \noindent Estimate for $S_1$.} It is not hard to show that

$$\sum_{k\in \Z} \exp(-\sqrt {x|k|})<\frac{5}{x} \mbox{ for $0<x<
1$}.$$

To see this, observe that

$$\sum_{k\ge 1} \exp(-\sqrt{xk})\le
\int_{0}^{\infty}\exp(-\sqrt{xt})dt=\frac{2}{x},$$

\noindent where the integral is evaluated by
changing variable and integration by parts.

\noindent Thus

\begin{equation}\label{equation:lemma3:decay}
\sum_{|k|\ge k_0} \exp(-\sqrt {x|k|})< \sum_{k\in \Z}
\exp(-\sqrt x(\frac{ \sqrt {|k|} +\sqrt {k_0}}{2})) \le
\frac{20}{x}\exp(-\frac{\sqrt{xk_0}}{2}).
\end{equation}

\noindent From the property of $f$ (Lemma \ref{lemma:Fourier}) we can deduce that

$$S_1 \le 16 \widehat{f}(0)\sum_{|m|\ge \frac{\gamma_1q(\log q)^{\gamma_2}}{L}}
\exp(-\delta \sqrt {|Lm/q|}), $$

which, via (\ref{equation:lemma3:decay}) and since $q\ge 3$, implies

$$S_1 \le 16
\widehat{f}(0)\frac{20}{Lq^{-1}}\exp(-\frac{\delta(\gamma_1(\log q)^{\gamma_2}
  )^{1/2}}{2})
\le \widehat{f}(0)/4,$$

given that we choose
$\gamma_1, \gamma_2$ sufficiently large.

{ \noindent \it Estimate for $S_2$.} We have

$$S_2 =\frac{\widehat{f}(0)}{q}\sum_{\scriptstyle |m|\le L' \atop \scriptstyle m\neq 0}|\sum_{z=1}^{q}e(\frac{\bar{a}p'z^2}{q'} +\frac{2\bar{a}pz_1zm}{q})|.$$

\noindent We shall choose $D>\gamma_2$.

\noindent Set

$$\gamma_1:= \big(\frac{6(D-\gamma_2)}{e}\big)^{D-\gamma_2}.$$

\noindent First, we observe that

$$L'q = \frac{2\gamma_1 q^2(\log q)^{\gamma_2}}{(sq)^{1/2}(\log q)^{D}} =
\frac{2\gamma_1 q^{3/2}}{s^{1/2}(\log q)^{D-\gamma_2}}
=\frac{2\gamma_1{q'}^{1/2}q}{(\log q)^{D-\gamma_2}}\ge
{q'}^{4/3}\frac{\gamma_1q^{1/6}}{(\log q)^{D-\gamma_2}}.$$

\noindent It is not hard to show that the function
$q^{1/6}/(\log q)^{D-\gamma_2}$, where $q\ge 3$, attains its
minimum at $q=\exp(6(D-\gamma_2))$. Therefore, by the choice of $\gamma_1$, we have

$$L'q \ge {q'}^{4/3}.$$

\noindent Next, Lemma \ref{lemma:Weyl} applied for $\ep=1/3$ (and with
the mentioned $c$ and $\alpha$) yields

\begin{align*} S_2 &=\frac{\widehat{f}(0)}{q}\sum_{\scriptstyle |m|\le L' \atop \scriptstyle m\neq 0}|\sum_{z=1}^{q}e(\frac{\bar{a}p'z^2}{q'} +\frac{2\bar{a}pz_1zm}{q})| \\
& \le c \frac{\widehat{f}(0)}{q}(\frac{L'q}{\sqrt{q'}}+
L'\sqrt{q})(\log q)^{\alpha} \\
&\le 2c \frac{\widehat{f}(0)}{q}\frac{L'q}{\sqrt{q'}}(\log q)^{\alpha}= 2c\frac{\widehat{f}(0)L'}{\sqrt{q'}}(\log q)^{\alpha}.\end{align*}

\noindent It follows that

$$S_2 \le\frac{4c\gamma_1 q(\log q)^{\alpha+\gamma_2}}{(\sqrt{sq}\log^D
q)\sqrt{q'}} \widehat{f}(0) = \frac{4 c\gamma_1 (\log q)^{\alpha+ \gamma_2} }
{(\log q)^D} \widehat{f}(0).$$

\noindent Now we choose $D,\gamma_2$ so that $D-\gamma_2-\alpha
=1$. Thus $\gamma_1=(6(\alpha+1)/e)^{\alpha+1}$, and

$$S_2\le \frac{4c \gamma_1(\log q)^{\alpha+\gamma_2}}{(\log q)^D}\widehat{f}(0) =
\frac{4c(6(\alpha+1)/e)^{\alpha+1}}{\log q} \widehat{f}(0)\le \widehat{f}(0)/4$$

\noindent where the last inequality comes from (\ref{assumeq}).

\end{proof}

\begin{remark}

\noindent We can also use Burgess estimate to have an alternative proof with a slightly better bound. However, an improvement in this section does not improve the main theorem.
\end{remark}

\section{Proof of Lemma \ref{lemma:main2}}\label{section:lemma2}

\noindent We first apply Lemma \ref{lemma:main1} to obtain a large proper GAP
$Q$ of rank 1 or 2.  By this lemma, we  have $A^{''} \subset \{s_1, \dots, s_m \} + Q $, where $m$ is a constant.

\noindent Let $S_i=A''\cap (s_i+Q)$ for $1\le i \le m$. We would like to guarantee that all $S_i$ are large by the following argument.

 If $S_i$ is smaller than $n^{1/3}(\log n)^{3C/10} $, then we delete it from $A^{''}$ and add to $A'$.  The new sets $A'$, $A^{''}$ and $Q$ still satisfy the claim of Lemma \ref{lemma:main1}. On the other hand,
 that the total number of elements added to $A'$ is only
 $O(n^{1/3}(\log n)^{3C/10} = o(|A'|)$, thus the sizes of $A'$ and $A''$
hardly changes.

From now on, we assume that
$|S_i|\ge n^{1/3}(\log n)^{3C/10}$ for all $i$.

For convenience, we let

$$s_i':=s_i+r.$$

Thus every element of $S_i$ is congruent with $s_i'$ modulo $q$.

\subsection{$Q$ has rank one} In this subsection, we deal with the (easy) case when $Q$ has rank one. We write
 $Q=\{r+qx \,\,| 0\le x\le L\}$ where $L \ge n^{2/3} (\log n) ^{C/2}$.

\noindent Since $Q\subset S_{A'} \subset [\frac{n}{p} |A'|]$, we have

$$q \le \frac{|A'|n}{pL}  \le \frac{n^{2/3}}{(\log n)^{C/6} p}.$$

\noindent By setting  $C$ (of Lemma \ref{lemma:main1}) sufficiently large
compared to $D$ (of Lemma \ref{lemma:main3}), we can guarantee that

\begin{equation}\label{equation:lemma2:dim1:1}
(pq)^{1/2}(\log q)^D\le n^{1/3}.
\end{equation}

\noindent Let  $d:=(s_1+r,\dots,s_m+r,q)=(s_1',\dots,s_m',q)$. If $d>1$ then all elements of $A''$ are divisible by
$d$, since $A^{''}$ are covered by $\{s_1, \dots, s_m\} +Q$. Thus we reach the third case of the lemma and are done.

 Assume now that $d=1$. By Lemma \ref{lemma:main3}, we can find
$0\le x_i \le (pq)^{1/2}(\log q)^D$ such that

\begin{equation}\label{equation:lemma2:dim1:2}
s_1'x_1+\dots+s_m'x_m+r \equiv pz^2 (\bmod{q}).
\end{equation}

\noindent Pick from $S_i$'s exactly $x_i$ elements and add them
together to obtain a number $s$. The set $s+ Q$ is a translate of $Q$ which
satisfies the first case of Lemma \ref{lemma:main2} and we are done.

\subsection{$Q$ has rank two} In this section, we consider the (harder) case when $Q$ has rank two. The main idea is similar to the rank one case, but the technical details are somewhat more tedious. We write

 $$Q=r+q(q_1x+q_2y)| 0\le x
\le L_1, 0\le y \le L_2$$

 where $L_1L_2 =|Q|  \ge n
 \log^{2C/3} n$.

\noindent  As $Q$ is proper, either $q_1\ge L_2$ or $q_2\ge L_1$ holds. Thus $qL_1L_2
\le |A'|n/p$, which yields (with room to spare)

\begin{equation}\label{equation:lemma2:dim2:1}
q\le \frac{n^{1/3}}{(\log n)^{C/6} p} .
\end{equation}

\noindent We consider two cases. In the first (simple) case, both $L_1$ and $L_2$ are large. In the second, one of them can be small.

{\bf Case 1.} $\min(L_1,L_2)\ge n^{1/3}(\log n)^{C/4}$.  Define $d:=(s_1',\dots,s_m',q)$ and argue as in the previous section. If $d>1$, then we end up with the third case of Lemma \ref{lemma:main2}. If $d=1$
then
apply Lemma \ref{lemma:main3}. The fact that $q$ is sufficiently small (see
(\ref{equation:lemma2:dim2:1}))  and that $|S_i|$ is sufficiently large
guarantee that we can choose $x_i$ elements from $S_i$. At the end, we will obtain a GAP of rank 2 which is a translate of $Q$ and satisfies the second case of Lemma \ref{lemma:main2}.

{\bf Case 2.} $\min(L_1,L_2)\le n^{1/3} (\log n)^{C/4}$. In this case the sides of GAP $Q$ are unbalanced and one of them is much larger than the other. We are going to exploit this fact to create a GAP of rank one (i.e., an arithmetic progression) which satisfies the first case of Lemma \ref{lemma:main1}, rather than trying to create a GAP of rank two as in the previous case.

\noindent Without loss of
generality, we assume that $L_1\le n^{1/3}(\log n)^{C/4}$. By the lower bound on
$L_1L_2$, we have that $L_2\ge
n^{2/3}(\log n)^{C/4}$. This implies

  $$qq_2 \le \frac{|A'|n}{pL_2 } \le
\frac{n^{2/3}}{(\log n)^{C/12} p}. $$ Again by setting $C$ sufficiently large compared to $D$, we have

\begin{equation}\label{equation:lemma2:dim2:2}
(pqq_2)^{1/2}(\log qq_2)^D\le n^{1/3}(\log n)^{C/5}.
\end{equation}

{\it Creating a long arithmetic progression.} In the rest of the proof we
make use of $A''$ and $Q$ to create an AP of type $\{r'+qq_2x_2\,\,|0\le x_2 \le L_2,
r'\equiv pz^2(\bmod{qq_2})\}$. This gives the first case in Lemma \ref{lemma:main1} and  thus completes the proof of this lemma.

\noindent Let $S$ be an element of $\{S_1,\dots,S_m\}$. Since $S$ is
contained in a translate of $Q$, there is a number $s$ such that
any $a\in S$ satisfies $a
\equiv s+tqq_1(\bmod{qq_2})$ for some $0\le t \le L_1$ (for instance,
if $a\in S_i$ then $a \equiv s_i'+tqq_1(\bmod{qq_2})$). Let $T$ denote the {\it multiset} of $t$'s obtained this way. Notice that
$T$ could contain one element of multiplicity $|S|$. Also recall that
$|S| \ge n^{1/3} (\log n)^{3C/10}$.

\noindent For $0\le l \le |S|/2$, let $m_{l}$ and $M_{l}$ (respectively) be the
minimal and  maximal values of the sum of  $l$ elements
of $T$. Since $0\le t\le L_1$ for every $t\in T$, by swapping summands of
$m_l$ with those of $M_l$, we can obtain a sequence
 $m_l = n_0 \le \dots \le n_{l} =M_l$ where each $n_i \in l^\ast T$ and $ n_{i+1}-n_i \le L_1$  for all relevant $i$.

\noindent By construction, we have

\begin{equation}\label{equation:lemma2:dim2:2'}
[m_l,M_l] \subset \{n_0, \dots, n_l \} + [0,L_1] \subset l^\ast T + [0,L_1].
\end{equation}

\noindent Next we observe that if  $l$ is large and $M_l-m_l$ is
small, then $T$ looks like a sequence of only one element with high
multiplicity.  We will call this element the  {\it essential} element
of $T$.

\begin{proposition}\label{fact:lemma2:1}
Assume that $\frac{1}{4} (n^{1/3}(\log n)^{3C/10}\le l \le
\frac{1}{2} n^{1/3}(\log n)^{3C/10}$ and $M_l-m_l < \frac{1}{4} n^{1/3} (\log n)^{3C/10}$.
Then all but at most $\frac{1}{2} n^{1/3}(\log n)^{3C/10}$ elements of $T$ are
the same.
\end{proposition}

\begin{proof}(Proof of Proposition \ref{fact:lemma2:1})
Let $t_1\le t_2\le \dots \le t_l$ be the $l$ smallest elements of $T$
 and $t_1' \le \dots \le t_l'$ be the $l$ largest. By the upper bound on $l$ and lower bound on $|S|=|T|$, $t'_1 \ge t_l$.  On the other hand, $M_l-m_l=(t_1'-t_1)+\dots+(t_l'-t_l)$. Thus if $M_l-m_l <
\frac{1}{4} n^{1/3}(\log n)^{3C/10} \le l-1$ then $t_i'=t_i$ for some $i$. The
claim follows.
\end{proof}

\noindent The above arguments work for any $S$ among $S_1, \dots, S_m$. We now associate to each $S_i$ a multiset $T_i$, for all $1\le i \le m$.

{\bf Subcase 2.1} {\it The hypothesis in Proposition
\ref{fact:lemma2:1} holds for all $T_i$.} In this case we move to $A'$
those elements of $S_i$ whose corresponding parts in $T_i$ is not the
essential element. The number of elements moved is only
$O(n^{1/3}(\log n)^{3C/10})$, which is negligible compared to both $|A'|$ and $|A''|$. Furthermore, the properties claimed in Lemma \ref{lemma:main1} remain unchanged and the size of new $S_i$ are now at least $\frac{1}{2} n^{1/3} (\log n)^{3C/10}. $

\noindent Now consider the elements of  $A''$ with respect to
 modulo $qq_2$. Since each $T_i$ has only the essential element, the elements of $A^{''}$ produces at most  $m$
residues $u_i=s_i'+t_iqq_1$, each of multiplicity at least

$$|S_i|\ge
\frac{1}{2} n^{1/3}(\log n)^{3C/10} \ge (pqq_2)^{1/2}(\log qq_2)^D $$
where the last
inequality comes from (\ref{equation:lemma2:dim2:2}). Define
$d=(u_1,\dots,u_m,qq_2)$ and proceed as usual,
applying Lemma \ref{lemma:main3}.

{\bf Subcase 2.2} {\it The hypothesis in Proposition
\ref{fact:lemma2:1} does not hold for all $T_i$.} We can assume that, with respect to $T_1$, $M_{l}-m_{l} \ge
\frac{1}{4} n^{1/3}(\log n)^{3C/10} $ for all  $\frac{1}{4} n^{1/3} (\log n)^{3C/10}  \le l
\le \frac{1}{2} n^{1/3}(\log n)^{3C/10}$. From now on, fix an $l$ in this interval.

\noindent Next, for a technical reason, we  extract from $S_1$ a very small
part $S_1'$ of cardinality $n^{1/3}(\log n)^{C/5}$ and set $S_1^{''}=S_1\backslash S_1'$.  Let $T$ be the multiset associated with $S_1^{''}$. We can assume that $T$ satisfies the hypothesis of this subcase.

\noindent Define $d:=(s_1',\dots,s_m',q)$. As usual, the case $d>1$ leads to the third case of Lemma \ref{lemma:main2}, so we can assume $d=1$. By Lemma \ref{lemma:main3}, there exist integers

$$0\le
x_i\le (pq)^{1/2}(\log n)^D \le n^{1/3}(\log n)^{C/5} \le |S_i|$$ and $k,z_1$ such
that

\begin{equation}\label{equation:lemma2:dim2:3}
s_1'x_1+\dots+s_m'x_m+(ls_1'+r)=pz_1^2+kq.
\end{equation}

\noindent For $i\ge 2$ we pick from $S_i$ exactly $x_i$ elements
$a^i_1,\dots,a^i_{x_i}$, and for $i=1$ we pick $x_1$ elements
$a^1_1,\dots,a^1_{x_1}$ from $S_1'$ and add them
together. By (\ref{equation:lemma2:dim2:3}) the following holds for
some integer $k'$,

\begin{equation}\label{equation:lemma2:dim2:3'}
\sum_{i=1}^m \sum_{j=1}^{x_i}a^i_j + (ls_1'+r)=pz_1^2+k'q.
\end{equation}

\noindent Furthermore, by Proposition \ref{proposition:square:local:1},
as $q=(qq_1,qq_2)$, there exist $0\le x\le (pqq_2)^{1/2}\log^D(qq_2)$
and $k'',z_2$ such that

$$qq_1x+pz_1^2+(k'+m_lq_1)q=pz_2^2+ k''qq_2,$$

\begin{equation}\label{equation:lemma2:dim2:4}
pz_1^2+k'q+(x+m_l)qq_1=pz_2^2+k''qq_2.
\end{equation}

\noindent As $(pqq_2)^{1/2}\log^D(qq_2)\le n^{1/3} \log^{C/5} n$  and $n^{1/3}\log^{C/5} n\le M_l-m_l$, we
have

$$m_l \le x+m_l \le M_l.$$

\noindent On the other hand, recall that $[m_l,M_l]\subset l^{\ast}
T+[0,L_1]$ (see  (\ref{equation:lemma2:dim2:2'})), we have

$$\{ls_1'+r+[m_{l},M_{l}]qq_1\} \subset l^\ast S_1^{''}+r+[0,L_1]qq_1 (\bmod
	  {qq_2}).$$

\noindent Thus

\begin{equation}\label{equation:lemma2:dim2:4'}
ls_1'+r+(x+m_l)qq_1 \in l^\ast S_1^{''}+ r +
[0,L_1]qq_1(\bmod{qq_2}).
\end{equation}

\noindent Combining
(\ref{equation:lemma2:dim2:3'}),(\ref{equation:lemma2:dim2:4}) and
(\ref{equation:lemma2:dim2:4'}) we infer that there exist $l$ elements $a_1,\dots,a_l$ of $S_1^{''}$,
and there exist $0\le u \le L_1$ and $v$ such that

$$\sum_{i=1}^m\sum_{j=1}^{x_i}a^i_j+a_1+\dots+a_l + r+uqq_1 = pz_2^2+ vqq_2.$$

\noindent Hence, $ \sum_{i=1}^m\sum_{j=1}^{x_i}a^i_j+a_1+\dots+a_l + Q $
contains the AP $\{(pz_2^2+ vqq_2)+ qq_2x_2|0\le x_2 \le L_2\}$, completing Subcase 2.2.

\noindent Finally, one checks easily that the number of elements of $A''$
involved in the creation of $pz_2^2$ in all cases is bounded by
$O(n^{1/3}\log^{C/5}n)=o(|A'|)$, thus we may put all of them to
$A'$ without loss of generality.

\section{Proof of Theorem \ref{theorem:square:p}: The rank one case.}\label{section:dim1}

\noindent Here we consider the (easy) case when $Q$ (in Lemma \ref{lemma:main2}) has rank one. In this case, $S_{A'}$ contains an AP $Q=\{r+qx|0\le x\le
L\}$, where $L\ge n^{2/3}(\log n)^{C/4}$ as in the first statement of Lemma \ref{lemma:main2}. We are going to show that $Q$ contains a number of the form $pz^2$.

\noindent Write $r = pz_0^2 + tq$ for some $0\le z_0 \le q$. Since $r$ is a
sum of some elements of $A'$, we have

$$0\le r \le |A'| (n/p) \le \frac{n^{4/3}(\log n)^{C/3}}{p}.$$

\noindent Thus

\begin{equation}\label{equation:square:case1:t}
-pq \le t \le \frac{n^{4/3}(\log n)^{C/3}}{pq}.
\end{equation}

\noindent The interval $[t/pq,(t+L)/pq]$ contains at least two squares because

$$(\frac{L}{pq})^2 \ge \frac{n^{4/3}(\log n)^{C/2}}{(pq)^2} \ge
10\frac{t}{pq}+20.$$

\noindent Thus, we can find an integer $x_0\ge 0$ such that
$\frac{t}{pq} < x_{0}^{2 } < (x_{0} +1)^{2 } \le  \frac{t+L}{pq}.$
 It is implied that (since $0\le z_0\le q$)

\begin{equation} \label{eqn:sandwich} t\le pqx_0^2 + 2pz_0 x_0 \le
t+L.\end{equation}

\noindent Set $z:=z_0+qx_0$. We have

$$pz^2 = pz_0^2 + q(pqx_0^2 + 2pz_0 x_0).$$

On the other hand, by (\ref{eqn:sandwich}), the right hand side belongs to

$$ pz_0^2 + q[t,t+L] = pz_0^2 + tq + q[0,L]= r+q[0,L]=Q.$$

Thus, $Q$ contains $pz^2$, completing the proof for this case.

\section{Proof of Theorem \ref{theorem:square:p}: The rank two case }\label{section:dim2}

In this case, we assume that $S_{A'}$ contains a proper GAP as in the second statement of
Lemma \ref{lemma:main2}. We can write

$$Q=\{r+q(q_1x_1+q_2x_2)\,\,|0\le x_1\le L_1,0\le x_2\le
L_2,(q_1,q_2)=1\}$$

where

\begin{itemize}
\item $\min(L_1,L_2)\ge n^{1/3}(\log n)^{C/4},$

\vskip .1in

\item $L_1L_2 \ge n (\log n)^{C/2},$

\vskip .1in

\item $q\le \frac{n^{1/3}(\log n)^{-C/6}}{p},$

\vskip .1in

\item and $r=pz_0^2+ tq$ for some integers $t$ and $0\le z_0\le q$.
\end{itemize}

\noindent Since $r$ is a sum of some elements of $A'$, we have $0\le r \le
\frac{n^{4/3}(\log n)^{C/3}}{p}$, and so

$$-pq \le t \le \frac{n^{4/3}(\log n)^{C/3}}{pq}.$$

\noindent Without loss of generality, we assume that $q_2L_2\ge q_1L_1.$
Because $Q$ is proper, either $q_2\ge L_1$ or $q_1\ge L_2$. On the
other hand, if $q_2 < L_1$ then $L_2 \le q_1$, which is
impossible by the assumption. Hence,

$$q_2\ge L_1.$$

\noindent Now we write $Q=\{pz_0^2+q(q_1x_1+q_2x_2+t)|0\le x_1\le L_1,0\le
x_2\le L_2,(q_1,q_2)=1\}$ and notice that if we set $w :=z_0+zq$ then

$$pw^2 - pz_0^2 = p(z_0+qz)^2-pz_0^2 = q(pqz^2 + 2pz_0 z).$$

\noindent Thus if there is an integer $z$ satisfies

\begin{equation}\label{equation:square:case2:root}
pqz^2+2pz_0z \in \{q_1x+q_2y+t | 0\le x \le L_1,0\le y \le L_2\}
\end{equation}

\noindent then $pw^2 \in Q$, and we are done with this case.
The rest of the proof is the verification of the following proposition, which shows the existence of a desired $z$.

\begin{proposition}\label{proposition:square:case2}
There exists an integer $z$ which satisfies
(\ref{equation:square:case2:root}).
\end{proposition}

\begin{proof} (Proof of Proposition \ref{proposition:square:case2}) The
method is similar to that of Lemma \ref{lemma:main3}, relying on Poisson summation.

\noindent Set $a:=pq$ and $b:=2pz_0.$ Notice that since $0\le z_0 \le q$,  $0\le b\le 2pq=2a$. Our task is to  find a $z$ such that

$$az^2+bz-q_1x-t=q_2y \mbox{ for some } 0\le x\le L_1,0\le y\le L_2.$$

Define (with foresight; see (\ref{eqn:I})) $I_x:=[L_1/8, L_1/4]$ and

$$I_z:=[(\frac{q_1L_1/4+t}{a})^{1/2}+1, (\frac{q_2L_2+q_1L_1/8+t}{a})^{1/2}-1].$$ (Notice the that
the lower bounds on $L_1,L_2$ and the upper bound on $pq$ guarantee that the expressions under the square roots are positive.)

\noindent Since $r +qq_1L_1 +qq_2 L_2 =  pz_0^2 +tq + q(q_1L_1+q_2 L_2) \in Q$, it follows that (with $\max (Q)$ denoting the value of the largest element of $Q$)

$$q_2L_2+q_1L_1/8 +t \le \max(Q)/q \le  \frac{ p^{-1} n^{4/3}(\log n)^{C/3} }{q} = \frac{n^{4/3}(\log n)^{C/3}}{a}.$$

\noindent Thus

$$|I_z|\ge \frac{1}{4}\frac{(q_2L_2-q_1L_1/4)a^{-1}}{\sqrt{\frac{q_2L_2+q_1L_1/8+t}{a}}}$$

\begin{equation}\label{equation:square:case2:Iz}
|I_z| =\Omega ( \frac{q_2L_2}{n^{2/3}(\log n)^{C/6}}).
\end{equation}


\noindent By the definitions of $I_x$ and $I_z$, we have,
 for any $x\in I_x$ and $z\in I_z$

\begin{equation} \label{eqn:I} 0\le az^2+bz-q_1x-t \le a(z+1)^2-q_1x-t \le q_2L_2.\end{equation}

\noindent Thus, for any such pair of $x$ and $z$, if $az^2+bz-q_1x-t$ is divisible by
$q_2$, then $y :=(az^2+bz-q_1x-t)/q_2$ is an integer in $[1,L_2]$.
We are now using the ideas from Section
\ref{section:lemma3}, with respect to modulo $q_2$ and the intervals $I_x$, $I_z$.

\noindent Let $\bar{q_1}$ be the reciprocal of $q_1$ modulo $q_2$
(recall that $(q_1,q_2)=1$). Let $f$ be a
function given by Lemma \ref{lemma:Fourier} with respect to the interval $I_x$. For a given $z\in I_z$, the number of $x\in I_x$ satisfying
(\ref{equation:square:case2:root}) is at least $N_z$, where

$$N_z :=\sum_{m\in \Z}f(\bar{q_1}az^2+\bar{q_1}bz-\bar{q_1}t+mq_2).$$

\noindent By applying Poisson summation formula (\ref{PSF}) and summing over $z$ in $I_z$
we obtain

$$N:=\sum_{z\in I_z}N_z=\sum_{m\in \Z}\frac{1}{q_2}\widehat{f}(\frac{m}{q_2})\sum_{z\in I_z} e(\frac{(\bar{q_1}az^2+\bar{q_1}bz-\bar{q_1}t)m}{q_2}).$$

\noindent It suffices to show that $N >0$. Similar to the proof of Lemma \ref{lemma:main3}, we will again show that the right hand side is dominated by the contribution at $m=0$. By triangle inequality, we have

$$|N-\frac{1}{q_2}\widehat{f}(0)|I_z|| \le \sum_{\scriptstyle m\in \Z \atop \scriptstyle m\neq 0}\frac{1}{q_2}|\widehat{f}(\frac{m}{q_2})||\sum_{z\in I_z} e(\frac{(\bar{q_1}az^2+\bar{q_1}bz-\bar{q_1}t)m}{q_2})|.$$

\noindent Let $\gamma$ be a sufficiently large constant  and
let

$$L' := \frac{8q_2 (\log q_2)^{\gamma}}{L_1}.$$

We have

$$|N-\frac{1}{q_2}\widehat{f}(0)|I_z|| \le S_1+S_2$$

where

$$S_1:= \sum_{|m|\ge
L'}\frac{1}{q_2}|\widehat{f}(\frac{m}{q_2})||\sum_{z\in I_z}
e(\frac{(\bar{q_1}az^2+\bar{q_1}bz-\bar{q_1}t)m}{q_2})|$$

and

$$S_2:=\sum_{\scriptstyle |m|\le L' \atop \scriptstyle m\neq 0}\frac{1}{q_2}|\widehat{f}(\frac{m}{q_2})||\sum_{z\in I_z}
e(\frac{(\bar{q_1}az^2+\bar{q_1}bz-\bar{q_1}t)m}{q_2})|.$$

To conclude the proof, we will show that both $S_1$ and $S_2$ are $o(\frac{\widehat{f}(0)|I_z|}{q_2})$.

{\it \noindent Estimate for $S_1$.} By the property of $f$,

$$S_1 \le \frac{\widehat{f}(0)|I_z|}{q_2}\sum_{|m|\ge \frac{8q_2(\log q_2)^{\gamma}}{L_1}}
\exp(-\delta \sqrt {|m L_1/(8q_2)|}).$$

\noindent By (\ref{equation:lemma3:decay}), and as $q_2$ is large
($q_2\ge L_1 > n^{1/3}$), the inner sum is $o(1)$, so

\begin{equation}\label{equation:square:case2:S1}
S_1 =o(\frac{\widehat{f}(0)|I_z|}{q_2})
\end{equation} as desired.

{\it \noindent Estimate for $S_2$.} Let $q'= (\bar{q_1}a, q_2)$. We can write

\begin{equation}\label{equation:square:case2:q'}
\bar{q_1}a = q'q_1' ,q_2=q'q_2' \mbox{ with } (q_1',q_2')=1.
\end{equation}

\noindent Then

$$S_2 \le \frac{\widehat{f}(0)}{q_2}\sum_{\scriptstyle |m|\le L' \atop \scriptstyle m\neq 0}|\sum_{z\in I_z}
e(\frac{q_1'mz^2}{q_2'}+\frac{(\bar{q_1}bz-\bar{q_1}t)m}{q_2})|.$$

\noindent By Lemma \ref{lemma:Weyl} there are absolute constants  $c,\alpha$
such that

$$S_2 \le c\frac{\widehat{f}(0)}{q_2}
\Big( L'\sqrt{|I_z|}(\log n)^{\alpha} + \frac{ L'|I_z|(\log n)^{\alpha}}{\sqrt{q_2'}} \Big).$$

To show that $S_2 = o(\frac{\widehat{f}(0)|I_z|}{q_2})$, it suffices to show that

\begin{equation} \label{bound1}
L'(\log n)^{\alpha} = o(\sqrt {|I_z|})
\end{equation}

\noindent and

\begin{equation} \label{bound2}
L'(\log n)^{\alpha} = o(q_2')
\end{equation}

To verify (\ref{bound1}), notice that  by (\ref{equation:square:case2:Iz}), we have

$$|I_z|L_1^2 =\Omega( \frac{L_1^2q_2L_2}{n^{2/3}(\log n)^{C/6}}).$$

\noindent Thus

$$\frac{|I_z|} {{L'}^2 (\log n)^{2\alpha}} =\Omega
(\frac{|I_z|L_1^2}{q_2^2 (\log n)^{2\alpha+2 \gamma}}
)
= \Omega \Big(\frac{L_1^2L_2^2}{L_2 q_2 n^{2/3} (\log n)^{C/6 +2\alpha  + 2 \gamma}} \Big). $$

\noindent Since  $(L_1L_2)^2 \ge (n (\log n)^{C/2})^2=n^{2}\log^C n$ and
$L_2q_2 = O(\max(Q)) =O ( p^{-1} n^{4/3}(\log n)^{C/3})$, the last formula is $\omega (1)$ if we set $C$ sufficiently large compared to $\alpha$ and $\gamma$. This proves (\ref{bound1}).

\noindent As a result,

$$\frac{\widehat{f}(0)}{q_2} L'\sqrt{|I_z|}(\log n)^{\alpha}=o(\widehat{f}(0)|I_z|/q_2).$$

  Now we turn to (\ref{bound2}). Recall that $q_2=q'q_2'$ and $q'=(\bar{q_1}a,q_2)=(a,q_2)$ (as $q_1$ and $q_2$ are co-primes). Thus

$$q_2'\ge \frac{q_2}{a}=\frac{q_2}{pq}.$$

To show (\ref{bound2}), it suffices to show that

$$
\frac{q_2}{pq} = \omega ({L'}^2 (\log n)^{2\alpha} ) $$

\noindent which (taking into account the definition of $L'$) is equivalent to

$$ q_2 L_1^2 =\omega (pq q_2^2 (\log n)^{2\alpha+2 \gamma} ). $$

Multiplying both sides with $L_2 q_2^{-1}$, it reduces to

$$  L_1^2 L_2 =\omega (pq q_2 L_2 (\log n)^{2\alpha+2 \gamma} ). $$

Now we use the fact that $qq_2L_2 =O(\max (Q)) = O(p^{-1} n^{4/3} (\log n)^{C/3})$ and the lower bounds $L_1L_2 \ge n  (\log n)^{C/2}$ and $L_1 \ge n^{1/3} (\log n)^{C/4} $. The claim follows by setting $C$ sufficiently large compared to $\alpha$ and $\gamma$, as usual. Our proof is completed. \end{proof}

{\bf Acknowledgements.} The authors would like to thank
Henryk Iwaniec for helpful discussions.


\begin{thebibliography}{99}


\bibitem{Alon} {N. Alon}, {\it Subset sums}, Journal of Number Theory, 27 (1987),
196-205.
\bibitem{AlonFreiman} {N. Alon } and {G. Freiman}, {\it On sums of subsets of a set of
integers}, Combinatorica, 8 (1988), 297-306.

\bibitem{Bilu} {Y. Bilu}, {\it Structure of sets with small sumset},
 Structure theory of set addition. Asterisque 258 (1999), xi, 77-108.

\bibitem{Erdos} {P. Erd\H{o}s}, {\it Some problems and results on combinatorial number
theory}, Proc. 1st. China Conference in Combinatorics (1986).

\bibitem {Fre} {G. Freiman}, {\it Foundations of a structural theory of set addition}, translated from the Russian, Translations of Mathematical Monographs, Vol 37. American Mathematical Society, Providence, R. I., 1973.

\bibitem{Green}{B. Green}, {\it An exposition on triples in
 Arithmetic progression},

http://www.dpmms.cam.ac.uk/\~~bjg23/papers/bourgain-roth.pdf.

\bibitem{GT} {B. Green and T. Tao},  {\it Compressions, convex geometry and the Freiman-Bilu theorem},  Q. J. Math. 57 (2006), no. 4, 495--504


 \bibitem{IKbook} { H. Iwaniec} and  {E. Kowalski}, {\it Analytic number theory},
American Mathematical Society, Colloquium publications, Volume 53.

\bibitem{Lipkin} {E. Lipkin} {\it On representation of $r-$powers by subset
sums}, Acta Arithmetica 52 (1989), 114-130.


\bibitem{Ruzsa} {I. Ruzsa}, {\it An analogue of Freiman's theorem in
group}, Structure theory of set addition. Asterisque 258 (1999),
323-326.

\bibitem{Sarkozy} {A. S\'ark\"ozy}, {\it Finite addition theorems, II}, Journal of Number
Theory, 48 (1994), 197-218.

\bibitem{SzemVu1} { E. Szemer\'edi} and { V. H. Vu }, {\it Long
arithmetic progressions in sumsets: Thresholds and Bounds}, Journal
of the A.M.S, 19 (2006), no 1, 119-169.

\bibitem{TVbook} {T. Tao} and {V. H. Vu}, {\it Additive combinatorics}, Cambridge University Press, 2006.



\end{thebibliography}
\end{document}